\documentclass[letterpaper, 11pt]{article}
\usepackage[utf8]{inputenc}
\usepackage[T1]{fontenc}
\usepackage{amsmath, amssymb, amsthm, stmaryrd, dsfont, mathrsfs}
\usepackage{fullpage}
\usepackage{graphicx}
\usepackage{tikz-cd}
\usepackage{longtable}
\usepackage{enumitem}
\usepackage{multicol}
\usepackage{color,soul}
\usepackage{float}
\usepackage[normalem]{ulem}
\usepackage{hyperref}

\newtheorem{thm}{Theorem}[section]  
\newtheorem*{thm*}{Theorem}
\newtheorem{lem}[thm]{Lemma}
\newtheorem{cor}[thm]{Corollary}
\newtheorem{prop}[thm]{Proposition}
\newtheorem{rem}[thm]{Remark}

\newtheorem{conj}[thm]{Conjecture}

\title{notes}

\begin{document}
\noindent
\begin{center}
\large \textbf{Denominators in the Farey Sequence, the Stern Sequence, and the Frobenius Problem with Relatively Prime Multipliers} \\ 
\vspace{.75cm}
\textbf{Grace Natalie Jaffe} \\
\vspace{.5cm}
\hrule 
\vspace{.75cm}
\end{center}

\normalsize
\noindent

Abstract: The Frobenius problem asks, given two positive, relatively prime integers $a$ and $b$, what is the maximum integer that cannot be expressed as a linear combination of $a,b$ using non-negative coefficients? This problem is resolved. However, if we add the restriction that the coefficients themselves must also be relatively prime, we create an open problem with connections to the Stern and Farey sequences and the Jacobsthal function. We prove that under this new restriction, a maximum inexpressible integer still always exists, and explore its values and properties as functions of $a,b$.

\section{Introduction}
\label{sec:Introduction}

In 1850, Eisenstein wrote to Stern concerning which denominators appear between $\frac{t}{a}, \frac{s}{b}$ consecutive in some Farey sequence, and how often \cite{eisen}. This inspired Stern to devise his \emph{diatomic array}. In this paper, we use this and techniques of analytic number theory to show that if $\gcd(a,b) = 1$, then there is a computable function $g(a,b)$, so that if $n > g(a,b)$, then a fraction $\frac{m}{n}$ eventually appears between $\frac{t}{a}$ and $\frac{s}{b}$. This translates into a variation on the Frobenius problem: suppose $a$ and $b$ are positive, relatively prime integers. If $n > g(a,b)$, then there exist positive, relatively prime integers $x$ and $y$ so that $n = ax + by$.

Recall that the \emph{Farey sequence of order $n$}, denoted $F_n$, is defined as ``the ascending series of irreducible fractions between 0 and 1 whose denominators do not exceed $n$'' \cite{hardy}. The Farey sequences have a rich structure, including the fact that if $\frac{t}{a}, \frac{s}{b}$ are consecutive terms in a Farey sequence of any order, then $as-bt=1$. Furthermore, the first term to appear between these terms will be $\frac{t+s}{a+b}$, in the Farey sequence of order $a+b$. 

The \emph{diatomic array} is defined as having first row $a, b$ and each subsequent row a repetition of the previous row with the sum of consecutive terms inserted between them, as shown in Figure \ref{fig:da}. We will borrow the notation of \cite{rez1} in denoting the entire diatomic array with first row $a,b$ by $Z(a,b)$. It is convenient to classify the general recurrence in terms of rows and positions within the diatomic array: let a term $Z(r,k;a,b)$ occur in the $k^{th}$ position from the left in the $r^{th}$ row down of the diatomic array $Z(a,b)$, where the counts for both $k$ and $r$ start at 0 \cite{rez1}. Under this system, the first row (consisting of an $a$ and a $b$ and no other entries) is row 0, and has an $a$ in position 0 and a $b$ in position 1. Rewriting the recurrence in this language, we see that each entry in an even position, say position $2k$, is equal to the entry in position $k$ in the row immediately above, and each entry in an odd position, say position $2k+1$, is equal to the sum of the entries in positions $k$ and $k+1$ in the row immediately above. 

We define the Stern sequence by
\[s(0) = 0, \ \ \  s(1) = 1; \ \ \ \ \ \ \ \ \ s(2n)=s(n), \ \ \ s(2n+1)=s(n)+s(n+1) \text{ for } n \geq 1. \ \ \ \cite{rez2} \]
The Stern sequence satisfies a number of interesting properties that relate to the diatomic array \cite{rez2}, including:
\begin{itemize}
    \item If $0 \leq k \leq 2^r$, then $Z(r,k;0,1)=s(n)$.
    \item $Z(r,k;a,b) = s(2^r-k)a+s(k)b$.
    \item $s(2^r n + k) = s(2^r-k) s(n) + s(k) s(n + 1)$ for $0 \leq k \leq 2^r$.
    \item $s(2^r)=1$ and $s(2^r-1)=r$.
\end{itemize}

\begin{figure}
	\centering
	\includegraphics[width=6.5in]{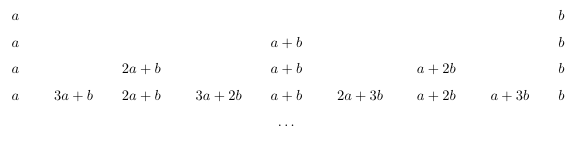}
	\caption{Rows 0-3 of the diatomic array $Z(a,b)$.}
	\label{fig:da}
\end{figure}

Eisenstein wrote Stern that he had a proof that every denominator between $\frac{t}{a}, \frac{s}{b}$ in some Farey sequence must appear in the diatomic array $Z(a,b)$ \cite{eisen}.

It is not difficult to show that $s:\mathbb{Z}_{\geq 0} \to \mathbb{Z}_{\geq 0}$ is surjective. In particular, we emphasize the point that every positive integer appears at least once somewhere within the diatomic array $Z(0,1)$. This will \emph{not} hold if we use the same recurrence relation but with a different choice of initial conditions, i.e. a different choice of $a,b$. One trivial way this condition can fail is if $\gcd(a,b) > 1$. It follows from the above that $Z(r,k;a,b)$ must be divisible by $\gcd(a,b)$ for all $r,k$. Then any natural number not divisible by $\gcd(a,b)$ cannot be expressed as $Z(r,k;a,b)$ and equivalently does not appear in $Z(a,b)$. However, adding in the condition that $\gcd(a,b)=1$ is not sufficient to conclude that every positive integer will appear in $Z(a,b)$. Suppose, for example, that we choose $a=5, b=7$. Then $\gcd(a,b)=1$, but we still cannot express every natural number as $Z(r,k;5,7)$ for some $r,k$. A cursory examination of $Z(5,7)$ makes clear that each ``new'' entry past row 0 must be at least 12:

\begin{align*}
&5 & \ & \ & \ & \ & \ & \ & \ & 7\\
&5 & \ & \ & \ & 12 & \ & \ & \ & 7\\
&5 & \ & 17 & \ & 12 & \ & 19 & \ & 7\\
&5 & 22 \ \ \ \ \ \ \ \ \ \ \ \  & 17 & 29 \ \ \ \ \ \ \ \ \ \ \ \ & 12 &  31\ \ \ \ \ \ \ \ \ \ \ \ \ & 19  & 26 \ \ \ \ \ \ \ \ \ \ \ \  & 7\\
&\ & \ & \ & \ & \ldots & \ & \ & \ & \ \\
\end{align*}

Then $1,2,3,4,6,8,9,10,11 \notin Z(5,7)$. Our chief goals here are to explore, for a general $a,b \in \mathbb{Z}^+$, which non-negative integers can be expressed as $Z(r,k;a,b)$ for some $r,k$, which cannot, and why. This work is drawn from my thesis, \textit{Topics in Analytic Number Theory}. In later sections, we will define $g(a,b)$ as the largest integer that cannot be expressed as $Z(r,k;a,b)$ and show not only that $g(a,b)$ always exists but also give an upper bound of $N(a,b)=p_1p_2 \cdots p_r$ for the minimum $r$ satisfying $ 2ab < \frac{p_3-1}{2} \cdot \frac{p_4-1}{2} \cdots \frac{p_r-1}{2} $ where $p_i$ is the $i^{th}$ prime (see Theorem \ref{thm: ta-da}). We will note $g(a,b)$ is bounded below by $f(a,b)$, the Frobenius number of $a$ and $b$, in Section \ref{sec: The Frobenius Connection}. We will focus on building better upper bounds in Theorem \ref{thm:fareyub} and Corollary \ref{cor:refinement}, and study patterns in the values of $g(a,b)$ in Sections \ref{sec: Exact Values} and \ref{sec: future}. First, we return our attention to the diatomic array at large.

\section{Definitions and Preliminary Results}
\label{sec: def,pr}

For any choice of $a,b$, we set \[S_{\left(a,b\right)} = \left\{ n \in \mathbb{Z}_{\geq 0} : \exists r,k \in \mathbb{Z}_{\geq 0} \text{ s.t. } n = Z(r,k;a,b) \text{ and } k \leq 2^r \right\} .\] Then $S_{\left(a,b\right)}$ is the set of exactly those non-negative integers appearing in the diatomic array $Z(a,b)$. For $\gcd(a,b)=d>1$, if $d|n$, we can verify
\[ n \in S_{\left(a,b\right)} \iff \frac{n}{d} \in S_{\left(\frac{a}{d},\frac{b}{d}\right)} \iff n \in d S_{\left(\frac{a}{d},\frac{b}{d}\right)}. \]
If $d \nmid n$, then $\frac{n}{d} \notin \mathbb{Z}_{\geq 0}$ and so $\frac{n}{d} \notin S_{\left(\frac{a}{d},\frac{b}{d}\right)}$; it follows from the above that $n \notin S_{\left(a,b\right)}$. Then $S_{\left(a,b\right)} = d S_{\left(\frac{a}{d},\frac{b}{d}\right)}$, which justifies restricting our consideration to cases where $\gcd(a,b)=1$.

We note that $Z(r,2^r-k;a,b) = Z(r,k;b,a)$ \cite{rez2}, and it follows that $S_{\left(a,b\right)} = S_{\left(b,a\right)}$, so we may treat the two interchangeably.

\begin{thm}
\label{thm: big buddy} There exist $x,y \in \mathbb{Z}_{\geq 0}$ such that $n = ax+by$ and $\gcd(x,y)=1$ if and only if $n \in S_{\left(a,b\right)}$.
\end{thm}

We make several smaller claims before proceeding to the proof of Theorem \ref{thm: big buddy}.

\begin{prop}
\label{prop: at least}
Let $r \geq 1$, $0 \leq k \leq 2^r$, and $a < b$. Then each new element in $Z(a,b)$ (i.e. every element in an odd position) is at least $ra+b$.  
\end{prop}

This can be proven by inducting on $r$, and less formally, follows intuitively from Figure $\ref{fig:da}$.
It follows as a result that every odd element in row $(m+1)$ of $Z(a,b)$ is greater than at least one odd element in row $m$ of $Z(a,b)$.

\begin{prop}
\label{prop: first r rows}
If $a<b$, $n < ra+b$, and $n$ does not appear in the first $r$ rows of $Z(a,b)$, then $n \notin Z(a,b)$.
\end{prop}

This time, the proof considers the $\left(r+t\right)^{th}$ row of $Z(a,b)$ for $t \geq 1$ and proceeds by induction on $t$.

\begin{thm}
\label{thm: relatively prime main}
Let $r \geq 0$ and $0 \leq k \leq 2^r$. Then $\gcd(s(k),s(2^r-k))=1$.
\end{thm}

\begin{proof}
We proceed by induction on $r$. Given that $\left\{s(0),s(1),s(2),s(3),s(4) \right\} = \left\{0,1,1,2,1\right\} $, it is easy to verify that the claim holds for $r=0,1,2$. Suppose now that the claim holds for some $r \geq 2$ and consider $\gcd(s(k),s(2^{r+1}-k) )$ for $0 \leq k \leq 2^{r+1}$. It suffices to consider $0 \leq k \leq 2^r$, for if $k \geq 2^r$, then by letting $l=2^{r+1}-k$, we have
\[ \gcd ( s(k), s(2^{r+1}-k) ) = \gcd (s(2^{r+1}-l),s(l) ) = \gcd (s(l),s(2^{r+1}-l)) \]
where $0 \leq l \leq 2^r$. Then we may assume without loss of generality that $0 \leq k \leq 2^r$, and consequently may put $k = 2^r-m$ for $0 \leq m \leq 2^r$. Then
\[ \gcd (s(k),s(2^{r+1}-k)) =\gcd (s(2^r-m),s(2^{r+1}-2^r+m) ) 
=\gcd (s(2^r-m),s(2^r+m)). \]
By recalling $s(2^r n + k) = s(2^r-k) s(n) + s(k) s(n +1)$ for $0 \leq k \leq 2^r$ \cite{rez2} and setting $n=1,k=m,$ we obtain
\[ s(2^r+m) =s(2^r \cdot 1 +m) = s(2^r-m)s(1)+s(m)s(2) = s(2^r-m)+s(m).\]
Then
\begin{align*}
\gcd (s(k),s(2^{r+1}-k)) &= \gcd (s(2^r-m),s(2^r+m))\\
&= \gcd (s(2^r-m),s(2^r-m)+s(m))\\
&= \gcd (s(2^r-m),s(m))\\
&= 1,
\end{align*}
where the last step follows by the induction hypothesis.
\end{proof}

\begin{cor}
\label{cor: relatively prime cor}
Let $r \geq 0$ and $0 \leq k \leq 2^r$. Then $\gcd( s(k),s(2^r+k) )=1$.
\end{cor}

\begin{proof}
Since we have seen that $s(2^r+k)=s(2^r-k)+s(k)$ in the proof of Theorem \ref{thm: relatively prime main}, it follows that
\[ \gcd( s(k),s(2^r+k)) = \gcd( s(k),s(2^r-k)+s(k)) = \gcd( s(k),s(2^r-k)) = 1 \]
where the last equality follows directly from Theorem \ref{thm: relatively prime main}.
\end{proof}

\begin{rem}
\label{rem: knuth}
    Let $a,b \in \mathbb{Z}_{\geq 0}$ with $\gcd(a,b)=1$. Then there exists a unique odd $k$ and $r$ such that $a=s(k)$ and $b=s(2^r+k)$.
\end{rem}

This follows immediately from the fact that ``we can...regard the Stern-Brocot tree as a number system for representing rational numbers, because each positive, reduced fraction occurs exactly once'' \cite[pg. 119]{knuth}.

We now have all the ingredients needed to prove our earlier claim.

\begin{proof}[Proof of Theorem \ref{thm: big buddy}]

$\left(\impliedby\right)$ Let $n \in S_{\left(a,b\right)}$. Then there exists some $r \geq 0$, $0 \leq k \leq 2^r$ such that 
\[n = Z(r,k;a,b) = s(2^r-k)a+s(k)b.\]
By Theorem \ref{thm: relatively prime main}, $\gcd (s(k),s(2^r-k))=1$, so we may take $x = s(k)$ and $y = s(2^r-k)$, which guarantees $x,y \in \mathbb{Z}_{\geq 0}$, $\gcd(x,y)=1$, and $n=ax+by$. 

$\left(\implies \right) $ Let there exist $x,y \in \mathbb{Z}_{\geq 0}$ such that $n = ax+by$ and $\gcd(x,y)=1$. Because $\gcd(x,y)=1$, we also know $\gcd(y,x+y)=1$, and so there exists a unique odd $k$ and $r$ such that $y=s(k)$ and $x+y=s(2^r+k)$ by Remark \ref{rem: knuth}. As we have seen above, $s(2^r+k)=s(2^r-k)+s(k)$, and since $y=s(k)$,
\[ x+y=s(2^r+k)=s(2^r-k)+s(k) \implies x=s(2^r-k). \]
Then $n = ax+by=s(2^r-k)a+s(k)b$ which implies $n \in S{\left(a,b\right)}$.
\end{proof}

Interpreted through the lens of Theorem \ref{thm: big buddy}, the problem begins to bear a strong resemblance to the \emph{Frobenius problem}.

\section{The Frobenius Connection}
\label{sec: The Frobenius Connection}

For a finite set of positive integers $\left\{a_1,a_2,\ldots,a_n\right\}$ such that $\gcd(a_1,a_2,\ldots,a_n)=1$, the \emph{Frobenius number} $f(a_1,a_2,\ldots,a_n)$ is the largest integer that \emph{cannot} be expressed as a sum of the form $k_1a_1+k_2a_2+\ldots+k_na_n$ for non-negative integers $k_1,k_2,\ldots,k_n$ \cite{frob}. The \emph{Frobenius problem} is to find the Frobenius number given information about the set $\left\{a_1,a_2,\ldots,a_n\right\}$. We note that $f(a,b)=ab-a-b$ was proven by James Joseph Sylvester in 1882 \cite{frob}.

For positive integers $a,b$ such that $\gcd(a,b)=1$, we define $g(a,b)$ to be the largest integer that cannot be expressed as $ax+by$ for \emph{relatively prime}, non-negative integers $x,y$; if no such integer exists, we will say $g(a,b)=\infty$. This is essentially a variation of the Frobenius problem in which we set $n=2$ and place restrictions on the coefficients $k_1,k_2$ in addition to those on the set elements $a_1,a_2$. Note that by the previous section, $g(a,b)$ is equivalently defined as $\sup\left(\mathbb{Z} \setminus S_{\left(a,b\right)}\right)$. 

Certainly $g(a,b) \geq f(a,b)$, as there cannot be a relatively prime integral solution to $ax+by=N$ unless there is already an integral solution. However, the existence of an integral solution, or of several integral solutions, is not sufficient; consider again our example in which $a=5, b=7$. By Section \ref{sec:Introduction}, we already know $g(5,7)\geq 11$, and invoking the Frobenius problem, we can further conclude $g(5,7) \geq f(5,7) = 5 \cdot 7 - 5 - 7 =23$. We now use the FrobeniusSolve function in Mathematica to generate all non-negative integral solutions to $5x+7y=156$ in Figure \ref{fig:example}. Although we have four possible solutions, none of them satisfy the condition that $\gcd(x,y)=1$. Our choice of 156 was not accidental; in the next sections, we will develop the tools necessary to show $g(5,7)=156$, and thus that for every integer $N > 156$, we can guarantee the existence of a relatively prime integral solution to $5x+7y=N$. For now, we return our attention to a general choice of $a,b$ and begin to build the results necessary to prove that $g(a,b)$ always exists (that is, it is always finite).

\begin{figure}
	\centering
	\includegraphics[scale=0.6]{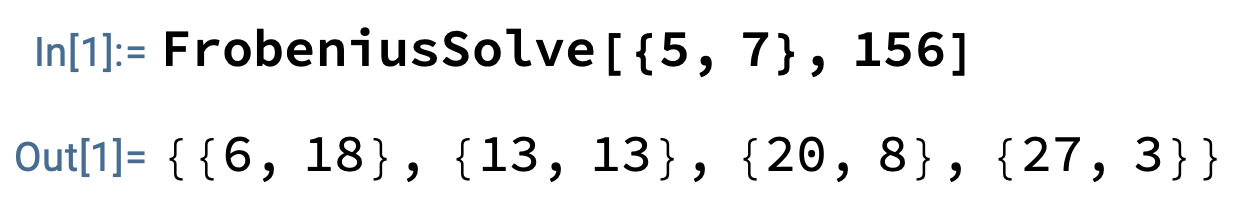}
	\caption{All non-negative integral solutions $\left(x,y\right)$ to $5x+7y=156$ have $\gcd(x,y)>1$.}
	\label{fig:example}
\end{figure}

\section{The Existence of \textit{g(a,b)}}
\label{sec: The Existence of g(a,b)}

So far, we have been able to offer an alternate definition and a lower bound for $g(a,b)$. We now prove that $g(a,b) < \infty$ for any choice of relatively prime $a,b$ using techniques from analytic number theory. 

\begin{lem}
\label{lem: ugh FINE}
For any choice of positive integers $a,b$, there exists a positive integer $r_0$ such that if $r \geq r_0$, then $ab \cdot 2^r < \left(p_1-1\right)\left(p_2-1\right)\cdots\left(p_r-1\right)$, where $p_i$ represents the $i^{th}$ prime.
\end{lem}

\begin{proof}
Note that
\begin{align*}
  &  ab \cdot 2^r < \left(p_1-1\right)\left(p_2-1\right)\cdots\left(p_r-1\right) \\
  \iff & ab < \frac{p_1-1}{2} \cdot \frac{p_2-1}{2} \cdot \frac{p_3-1}{2} \cdots \frac{p_r-1}{2} \\
  \iff & ab < \frac{2-1}{2} \cdot \frac{3-1}{2} \cdot \frac{p_3-1}{2} \cdots \frac{p_r-1}{2} \\
  \iff & 2ab < \frac{p_3-1}{2} \cdot \frac{p_4-1}{2} \cdots \frac{p_r-1}{2}.
\end{align*}
Because $a,b$ are fixed, we know $2ab$ is fixed, whereas the right-hand side of the last inequality grows without bound. Then there is certainly a minimum $r_0$ such that $2ab < \frac{p_3-1}{2} \cdot \frac{p_4-1}{2} \cdots \frac{p_{r_0}-1}{2}$, and for all $r \geq r_0$, we have
\[ 2ab < \frac{p_3-1}{2} \cdot \frac{p_4-1}{2} \cdots \frac{p_{r_0}-1}{2} \leq \frac{p_3-1}{2} \cdot \frac{p_4-1}{2} \cdots \frac{p_{r_0}-1}{2} \cdots \frac{p_r-1}{2}. \]
By the above, we may conclude $ab \cdot 2^r < \left(p_1-1\right)\left(p_2-1\right)\cdots\left(p_r-1\right)$ for all $r \geq r_0$.
\end{proof}

\begin{lem}
\label{lem: satan's little helper}
Let $\phi(n)$ be Euler's phi function and $\omega(n)$ be the number of distinct primes dividing $n$. Then for any choice of positive integers $a,b$ and all sufficiently large $n$, we have $ ab \cdot 2^{\omega(n)} < \phi(n)$.
\end{lem}

\begin{proof}
Because $a,b$ are positive integers, it follows by Lemma \ref{lem: ugh FINE} that there exists $r_0 \in \mathbb{Z}^+$ such that if $r \geq r_0$, then $ab \cdot 2^r < \left(p_1-1\right)\left(p_2-1\right)\cdots\left(p_r-1\right)$. We set $N=p_1p_2\cdots p_{r_0}$ and note by computing $\phi(N)$ and $\omega(N)$ directly that $ab \cdot 2^{r_0} < \left(p_1-1\right)\left(p_2-1\right)\cdots\left(p_{r_0}-1\right)$ if and only if $ab \cdot 2^{\omega(N)} < \phi(N)$.

Let $m > N$. We divide into two cases.

Case 1: $\omega(m) > \omega(N)$. We set $m=q_1^{l_1}q_2^{l_2}\cdots q_k^{l_k}$, with $q_1 < q_2 < \ldots < q_k$ all prime and each $l_i$ a positive integer. Since $k:=\omega(m)>\omega(N)=r_0$, it follows from Lemma \ref{lem: ugh FINE} that $ab \cdot 2^k < \left(p_1-1\right)\left(p_2-1\right)\cdots\left(p_k-1\right)$. We also observe $q_i \geq p_i$ for $i \leq k$ to obtain the following:
\begin{align*}
    \phi(m) &= \phi\left(q_1^{l_1}q_2^{l_2}\cdots q_k^{l_k}\right)= \prod_{i=1}^k \phi\left(q_i^{l_i}\right) = \prod_{i=1}^k q_i^{l_i-1}\left(q_i-1\right) \geq \prod_{i=1}^k q_i^0(q_i-1) \\&= \prod_{i=1}^k (q_i-1) \geq \prod_{i=1}^k (p_i-1) > ab \cdot 2^k.
\end{align*}

Hence $ ab \cdot 2^{\omega(m)} < \phi(m)$.

Case 2: $\omega(m) \leq \omega(N)$. Again set $m=q_1^{l_1}q_2^{l_2}\cdots q_k^{l_k}$, with $q_1 < q_2 < \ldots < q_k$ all prime and each $l_i$ a positive integer. It is still true that $q_i \geq p_i$ for $i \leq k$, and $q_i \geq p_i$ implies that $1-\frac{1}{q_i} \geq 1-\frac{1}{p_i}$. We also know that $ 1-\frac{1}{p_j} \leq 1$ for $j \leq r_0$, and together with an alternate formula for $\phi$, we see
\[ \phi(m) = m \prod_{i=1}^k \left( 1 - \frac{1}{q_i} \right) \geq m \prod_{i=1}^k \left( 1 - \frac{1}{p_i} \right) \geq m \prod_{i=1}^{r_0} \left( 1 - \frac{1}{p_i} \right) > N \prod_{i=1}^{r_0} \left( 1 - \frac{1}{p_i} \right) = \phi(N).\]

Since $\omega(m) \leq \omega(N)$ and we have already seen above that $ab \cdot 2^{\omega(N)} < \phi(N)$, it follows that
\[  ab \cdot 2^{\omega(m)} \leq ab \cdot 2^{\omega(N)} < \phi(N) < \phi(m). \]
This completes the second case, and so we may conclude that for all $m \geq N$, we have $ ab \cdot 2^{\omega(m)} < \phi(m)$. 
\end{proof}

\begin{thm}
\label{thm: ta-da}
For $a,b > 0$ with $\gcd(a,b)=1$, we have $g(a,b) < \infty$.
\end{thm}

\begin{proof}
Let $a,b \in \mathbb{Z}_{\geq 0}$ with $\gcd(a,b)=1$ be fixed. We show that for sufficiently large $n$,
\[ \left\lvert \left\{ \left(x,y\right) \in \mathbb{Z}_{\geq 0} \times \mathbb{Z}_{\geq 0} : ax+by=n \ \& \ \gcd(x,y)=1 \right\} \right\rvert \geq 1; \]
in particular, the set is non-empty. 

For $\mu\left(d\right)$ the Möbius function, recall $\sum_{d|n} \mu\left(d\right) = 1$ for $n=1$ and 0 otherwise. Then we may count $ \left\lvert \left\{ \left(x,y\right) \in \mathbb{Z}_{\geq 0} \times \mathbb{Z}_{\geq 0} : ax+by=n \ \& \ \gcd(x,y)=1 \right\} \right\rvert$ as
\begin{align*}
     & \hspace{0.19in} \sum_{\substack{x,y \geq 0 \\ ax+by=n}}  \hspace{0.15in} \sum_{d|\gcd(x,y)} \hspace{0.1in} \mu\left(d\right) \\
    &=  \sum_{\substack{x,y \geq 0 \\ ax+by=n}}  \hspace{0.31in} \sum_{d|x, d|y} \hspace{0.26in}  \mu\left(d\right)\\
    &=  \hspace{0.13in} \sum_{d \geq 1} \hspace{0.33in} \sum_{\substack{x,y \geq 0 \\ ax+by=n \\ d|x, d|y}} \hspace{0.13in} \mu\left(d\right)\\
    &=  \hspace{0.13in} \sum_{d \geq 1} \mu\left(d\right)  \sum_{\substack{x,y \geq 0 \\ ax+by=n \\ d|x, d|y}} \hspace{0.2in} 1.
\end{align*}

\noindent Upon setting $x=du, y=dv$ for $u,v \in \mathbb{Z}_{\geq 0}$, we obtain 
\begin{align*}
    \sum_{d \geq 1} \mu\left(d\right) \sum_{\substack{x,y \geq 0 \\ ax+by=n \\ d|x, d|y}} 1 &= \sum_{d \geq 1} \mu\left(d\right) \sum_{\substack{du,dv \geq 0 \\ adu+bdv=n }} 1\\
    &= \sum_{d \geq 1} \mu\left(d\right) \left\lvert \left\{\left(u,v\right) \in \mathbb{Z}_{\geq 0} \times \mathbb{Z}_{\geq 0} : adu+bdv=n \right\}\right)\rvert\\
    &= \sum_{d | n} \mu\left(d\right) \left\lvert \left\{\left(u,v\right) \in \mathbb{Z}_{\geq 0} \times \mathbb{Z}_{\geq 0} : au+bv=\frac{n}{d} \right\}\right\rvert.
\end{align*}

We now turn our attention to $\left\lvert \left\{\left(u,v\right) \in \mathbb{Z}_{\geq 0} \times \mathbb{Z}_{\geq 0} : au+bv=\frac{n}{d} \right\}\right\rvert$. Because we know that $\gcd(a,b)=1$ and $1 \mid \frac{n}{d}$ for all divisors $d$ of $n$, it follows that there always exists some solution to $au_0+bv_0=\frac{n}{d}$ for integers $u_0, v_0$ (not necessarily non-negative). All other integer solutions must be of the form $u=u_0+kb, v = v_0-ka$ for some $k \in \mathbb{Z}$. Since we want to count only those solutions for which $u,v \geq 0$, we count only those integers $k$ for which $u_0+kb \geq 0$ and $v_0-ka \geq 0$. Solving each for $k$, we obtain $\displaystyle \frac{-u_0}{b} \leq k$ and $k \leq \frac{v_0}{a}$. The number of integers $k$ satisfying both inequalities is then
\begin{align*}
    \Big\lfloor \frac{v_0}{a} \Big\rfloor - \Big\lceil \frac{-u_0}{b} \Big\rceil +1 &= \Big\lfloor \frac{v_0}{a} \Big\rfloor + \Big\lfloor \frac{u_0}{b} \Big\rfloor +1\\
    &= \frac{v_0}{a} - \left\{ \frac{v_0}{a} \right\} + \frac{u_0}{b} - \left\{ \frac{u_0}{b} \right\}    +1  \\
    &= \frac{au_0+bv_0}{ab} - \left\{ \frac{u_0}{b} \right\}  - \left\{ \frac{v_0}{a} \right\} +1 \\
    &= \frac{n}{dab} - \left\{ \frac{u_0}{b} \right\}  - \left\{ \frac{v_0}{a} \right\} +1.
\end{align*}

By definition, $\{x\}$, the fractional part of any number $x$, always satisfies $\{x\} \in [0,1)$, so
\[ 1 - \left\{ \frac{u_0}{b} \right\}  - \left\{ \frac{v_0}{a} \right\} \leq 1 -0-0=1 \]
and
\[ 1 - \left\{ \frac{u_0}{b} \right\}  - \left\{ \frac{v_0}{a} \right\} > 1 -1-1=-1. \]
It follows that
\[ \left\lvert \left\{\left(u,v\right) \in \mathbb{Z}_{\geq 0} \times \mathbb{Z}_{\geq 0} : au+bv=\frac{n}{d}  \right\} \right\rvert = \frac{n}{dab}+E\left(\frac{n}{d}\right)\] 
where $E\left(\frac{n}{d}\right)$ represents an error as a function of $n, d$ and $\left\lvert E\left(\frac{n}{d}\right) \right\rvert \leq 1$ for all $n, d$. We are justified in treating the error as a function of $\frac{n}{d}$ because $a,b$ are fixed, and $u_0, v_0$ depend on $a,b,\frac{n}{d}$.
With this result, we may write
\begin{align*}
 & \ \ \ \ \sum_{d | n} \mu\left(d\right) \left\lvert \left\{\left(u,v\right) \in \mathbb{Z}_{\geq 0} \times \mathbb{Z}_{\geq 0} : au+bv=\frac{n}{d}  \right\} \right\rvert \\
   &= \sum_{d | n} \mu\left(d\right) \left( \frac{n}{dab}+E\left(\frac{n}{d}\right) \right)\\
    &= \sum_{d | n} \mu\left(d\right)  \frac{n}{dab}+ \sum_{d | n} \mu\left(d\right) E\left(\frac{n}{d}\right).
    \end{align*}

The first sum simplifies nicely using an analytic number theory identity, given as Theorem 2.3 in \cite[pg. 26-27]{apostol}:

\[ \sum_{d | n} \mu\left(d\right)  \frac{n}{dab} = \frac{1}{ab} \left( \sum_{d | n} \mu\left(d\right)  \frac{n}{d} \right) = \frac{1}{ab} \left( \phi(n) \right) = \frac{\phi(n)}{ab}.\]

Using the fact that $\left\lvert E\left(\frac{n}{d}\right)\right\rvert\leq1$, we can bound the second sum:

\[ \sum_{d | n} \mu\left(d\right) E\left(\frac{n}{d}\right) \leq \left\lvert \sum_{d | n} \mu\left(d\right) E\left(\frac{n}{d}\right) \right\rvert \leq \sum_{d | n} \left\lvert \mu\left(d\right)\right\rvert \left\lvert E\left(\frac{n}{d}\right)\right\rvert \leq \sum_{d | n} \left\lvert \mu\left(d\right) \right\rvert = 2^{\omega(n)}.\]
This last equality follows from the fact that $\left\lvert \mu\left(d\right) \right\rvert=1$ if and only if $d$ is a product of distinct prime factors, and $\left\lvert \mu\left(d\right) \right\rvert=0$ otherwise. Then
\[\sum_{d | n} \mu\left(d\right) E\left(\frac{n}{d}\right) \geq -\left\lvert \sum_{d | n} \mu\left(d\right) E\left(\frac{n}{d}\right) \right\rvert\geq  -2^{\omega(n)} \]
and so
\[\sum_{d | n} \mu\left(d\right)  \frac{n}{dab}+ \sum_{d | n} \mu\left(d\right) E\left(\frac{n}{d} \right) \geq  \frac{\phi(n)}{ab} - 2^{\omega(n)}.\]
This shows that
\[ \left\lvert \left\{ \left(x,y\right) \in \mathbb{Z}_{\geq 0} \times \mathbb{Z}_{\geq 0} : ax+by=n \ \& \ \gcd(x,y)=1 \right\} \right\rvert \geq \frac{\phi(n)}{ab} - 2^{\omega(n)} \]
and by Lemma \ref{lem: satan's little helper}, for sufficiently large $n$, 
\[  \phi(n) > ab \cdot 2^{\omega(n)}  \iff \frac{\phi(n)}{ab} > 2^{\omega(n)}  \iff \frac{\phi(n)}{ab} - 2^{\omega(n)} > 0 . \]
Then for sufficiently large $n$,
\[ \left\lvert \left\{ \left(x,y\right) \in \mathbb{Z}_{\geq 0} \times \mathbb{Z}_{\geq 0} : ax+by=n \ \& \ \gcd(x,y)=1 \right\} \right\rvert \geq \frac{\phi(n)}{ab} - 2^{\omega(n)} > 0 \]
and since $\left\lvert \left\{ \left(x,y\right) \in \mathbb{Z}_{\geq 0} \times \mathbb{Z}_{\geq 0} : ax+by=n \ \& \ \gcd(x,y)=1 \right\} \right\rvert$ is necessarily an integer, it follows that for sufficiently large $n$,
\[ \left\lvert \left\{ \left(x,y\right) \in \mathbb{Z}_{\geq 0} \times \mathbb{Z}_{\geq 0} : ax+by=n \ \& \ \gcd(x,y)=1 \right\} \right\rvert \geq 1. \]
Then there exists $N \in \mathbb{Z}$ such that for all $n \geq N$, there is at least one solution to $ax+by=n$ such that $x,y \in \mathbb{Z}_{\geq 0}$ and $\gcd(x,y)=1$; hence, $g(a,b) < N$, which concludes the proof.
\end{proof}

The above lemma and theorem also yield an upper bound for $g(a,b)$: define $N(a,b)$ to be $p_1p_2 \cdots p_r$ for the minimum $r$ satisfying $ 2ab < \frac{p_3-1}{2} \cdot \frac{p_4-1}{2} \cdots \frac{p_r-1}{2} $. By the above, this $N$ has the property that if $n \geq N$, then $\frac{\phi(n)}{ab} - 2^{\omega(n)} > 0$, and hence $g(a,b) < N(a,b)$. We now have enough tools to begin computing exact values of $g(a,b)$.

\section{Exact Values}
\label{sec: Exact Values}

We first note that we can solve for $g(a,b)$ precisely if $a=1$ or $a=2$.

\begin{prop}
\label{prop: onesies}
We have $g(1,b)=b-1$ if $b \neq 2$ and $g(1,2)=0$.
\end{prop}

This is easily proven directly, but also follows intuitively from writing the first several rows of $Z(1,b)$.

\begin{prop}
If $b>1$ is odd, then $g(2,b) = 4b-2$.
\label{prop: twosies}
\end{prop}

\begin{proof}
We first claim $4b-2 \notin S_{\left(2,b\right)}$ for any $b \in \mathbb{Z}_{\geq 0}$. Suppose for contradiction that $4b-2 \in S_{\left(2,b\right)}$. It follows from Theorem \ref{thm: big buddy} that $4b-2=2x+by$ for some $x,y \in \mathbb{Z}_{\geq 0}$ such that $\gcd(x,y)=1$. We divide into four cases based on the possible values of $y$.

Case 1: $y$ is odd. Then $2x+by$ is odd, but since $4b-2$ is even, this contradicts $4b-2=2x+by$.

Case 2: $y=2$. Then $4b-2 = 2x+2b \implies x=b-1$, which implies $x$ is even. This contradicts $\gcd(x,y)=1$.

Case 3: $y \geq 4$. Then $2x+by \geq 2x+4b \geq 4b > 4b-2$; this contradicts $4b-2=2x+by$.

Case 4: $y=0$. Then 
\[ 4b-2=2x \implies x=2b-1 \implies 1= \gcd(x,y) = \gcd(2b-1,0)=2b-1 \implies b=1. \]
This contradicts the assumption that $b > 1$.

These cases account for all possible values of $y$, so it follows that there is no solution to $4b-2=2x+by$ such that $x,y \in \mathbb{Z}_{\geq 0}$ and $\gcd(x,y)=1$. Then $g(2,b) \geq 4b-2$ for $b>1$ odd.

We now claim that $g(2,b) < 4b-1$ by arguing that for all $n \geq 4b-1$, we have $n \in S_{\left(2,b\right)}$.

First, $4b-1 = b + 3b-1$ where $3b-1$ is even, so $4b-1=2 \cdot \frac{3b-1}{2} + b \cdot 1$ where $\frac{3b-1}{2}, 1 \in \mathbb{Z}_{\geq 0}$ and $\gcd \left(\frac{3b-1}{2},1\right)=1$. Then $4b-1 \in S_{\left(2,b\right)}$. Next, $4b = 2b+2b=2 \cdot b + b \cdot 2$ where $b,2 \in \mathbb{Z}_{\geq 0}$ and $\gcd(b,2)=1$ because $b$ is odd, so $4b \in S_{\left(2,b\right)}$ as well.

We now consider $4b+k$ for $k \in \mathbb{Z}^+$. We divide again into cases, this time based on $k \mod 4$, but will only argue the first explicitly.

Case 1: $k \equiv 1 \mod 4$. Then $4b+k=b+3b+k$ where $3b+k$ is even because $3,b,k$ are odd, so $4b+k= 2 \cdot \frac{3b+k}{2} + b \cdot 1$ where $\frac{3b+k}{2}, 1 \in \mathbb{Z}_{\geq 0}$ and $\gcd (\frac{3b+k}{2},1)=1$.

Cases 2, 3, and 4, where $k \equiv 2,3, \text{ or } 0 \mod 4$ respectively, follow similar arguments.

Then for all $k \in \mathbb{Z}^+$, $4b+k$ is expressible as $2x+by$ where $x,y \in \mathbb{Z}_{\geq 0}$ and $\gcd(x,y)=1$. Therefore $g(2,b) < 4b-1$ for all odd $b>1$, and together with the above, this proves $g(2,b)=4b-2$ for all odd $b>1$.
\end{proof}

We do not have closed formulae for $g(a,b)$ outside of these cases, but can still compute exact values by testing down from the upper bound defined in the previous section. We define the function UpperBound in Mathematica, with exact code given in Figure \ref{fig:code}. Given a choice of $a,b$, UpperBound first finds those $r$, beginning at $r=3$, for which $ 2ab \geq \frac{p_3-1}{2} \cdot \frac{p_4-1}{2} \cdots \frac{p_r-1}{2} $; we know that at most finitely many of these $r$ exist, and that the right-hand side of this inequality is strictly increasing as a function of $r$. If we then isolate the last $r$ for which the inequality holds and add one, we obtain the minimum $r$ satisfying $ 2ab < \frac{p_3-1}{2} \cdot \frac{p_4-1}{2} \cdots \frac{p_r-1}{2} $, at which point we set UpperBound$\left(a,b\right)$ equal to $p_1p_2 \cdots p_r$ for this $r$. This returns $N(a,b)$ precisely.

\begin{figure}
	\centering
	\includegraphics[width=6.5in]{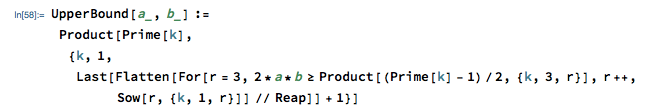}
	\caption{The Mathematica code of UpperBound, which computes the value $N(a,b)$.}
	\label{fig:code}
\end{figure}

We now wish to compare our upper bound to our known exact values. We note that by Proposition \ref{prop: onesies}, $g(1,999)=998$; by Proposition \ref{prop: twosies}, $g(2,999)=3994$; and that by direct computation, $N(1,999)=N(2,999)=9699690$. It is natural to ask next, is this the best possible upper bound, or can we build a better one?

\section{Computing an Alternate Upper Bound}
\label{sec: farey}
To answer this question, we return to the original problem posed in terms of the Farey sequence. We first show that for $\gcd(a,b)=1$, $a$ and $b$ will appear as neighboring denominators in some Farey sequence.

\begin{prop}
\label{prop: fareyexists}
For any choice of $a,b \geq 2$ such that $\gcd(a,b)=1$, there exist some $s,t$ such that $\frac{t}{a}, \frac{s}{b}$ are consecutive terms in some Farey sequence.
\end{prop}

\begin{proof}
Let $\gcd(a,b)=1$ with $a,b \geq 2$. Because $a,b$ are relatively prime, there necessarily exists some $s$ satisfying $as \equiv 1 \mod b$ and $1 \leq s \leq b-1$. Let $t=(as-1)/b$; then $t \geq 1$, and \[bt=as-1 \leq a(b-1)-1 = ab-a-1 < ab, \]
so $t < a$. Therefore $\frac{t}{a}, \frac{s}{b} \in (0,1)$. Additionally, it follows from $t=(as-1)/b$ that $as-bt=1$ and thus that $\gcd(a,t)=\gcd(b,s)=1$, so $\frac{t}{a}, \frac{s}{b}$ are irreducible. It also follows from  $as-bt=1$ that $\frac{s}{b}-\frac{t}{a}=\frac{1}{ab}$ and thus that $\frac{t}{a}, \frac{s}{b}$ are consecutive terms in $F_{\max(a,b)}:=F_k$ by \cite{ta2}.
\end{proof}

We now consider the denominators that appear in terms between $\frac{t}{a}, \frac{s}{b}$ in Farey sequence of sufficiently high order.

\begin{prop}
\label{prop: fairydemons}
If $\frac{t}{a}, \frac{s}{b}$ are consecutive terms in the Farey sequence of order $k:=\max(a,b)$, then $d$ appears as a denominator between $ \frac{t}{a}, \frac{s}{b}$ in some $F_{m}$ with $m>k$ if and only if $d=ax+by$ for some $x,y \in \mathbb{N}$ with $\gcd(x,y)=1$.
\end{prop}

\begin{proof}
Let $\frac{t}{a}, \frac{s}{b}$ be consecutive terms in the Farey sequence of order $k:=\max(a,b)$.

$(\implies)$ This direction was proved by Edward Sang in 1879 \cite{sang}. Essentially, every term between $\frac{t}{a}, \frac{s}{b}$ must be of the form $\frac{tx+sy}{ax+by}$ for $x,y \in \mathbb{N}$, and if $\gcd(x,y) > 1$, we find $\frac{tx+sy}{ax+by}$ is not in lowest terms. Thus every possible denominator $d$ appearing between $\frac{t}{a}, \frac{s}{b}$ must be of the form $ax+by$ with $\gcd(x,y)=1$.

$(\impliedby)$ Let $x,y \in \mathbb{N}$ with $\gcd(x,y)=1$ and $d=ax+by$. Since $\gcd(x,y)=1$, there exist $j,k \in \mathbb{Z}$ such that $jx+ky=1$. We know that ``if $ \frac{h}{k}$ and $ \frac{h'}{k'}$ are two successive terms of $F_n$, then $kh'-hk'=1$'' \cite{hardy}, so $as-bt=1$. Observe $(ak-bj),(js-kt) \in \mathbb{Z}$ and that
\begin{align*}
    & (ak-bj)(tx+sy)+(js-kt)(ax+by)\\
    = \ & aktx +aksy -bjtx - bjsy +ajsx +bjsy - aktx - bkty\\
    = \ & ajsx -bjtx + aksy - bkty  \\
    = \ & jx(as-bt) + ky(as-bt)\\
    = \ & jx+ky\\
    = \ & 1,
\end{align*}
and it follows that $\gcd(tx+sy,ax+by)=1$. Furthermore, it follows from $as-bt=1$ that $bt < as$, and so
\[ btx < asx \implies btx+bsy < asx+bsy \implies b(tx+sy) < s(ax+by) \implies \frac{tx+sy}{ax+by} < \frac{s}{b} \]
and
\[ bty < asy \implies atx+bty < atx+asy \implies t(ax+by) < a(tx+sy) \implies \frac{t}{a} <\frac{tx+sy}{ax+by}. \]
Therefore $\frac{tx+sy}{ax+by}=\frac{tx+sy}{d}$ appears somewhere between $\frac{t}{a}$ and $\frac{s}{b}$ in all Farey sequences of order at least $d$.
\end{proof}

\begin{cor}
\label{cor:immed1}
A number $d$ appears as a denominator between $\frac{t}{a}$ and  $\frac{s}{b}$ inclusive in some Farey sequence if and only if $d \in S_{\left(a,b\right)}$ as defined in Section \ref{sec: def,pr}.
\end{cor}

\begin{proof}
This follows immediately from Theorem \ref{thm: big buddy} and Proposition \ref{prop: fairydemons}.
\end{proof}

\begin{cor}
\label{cor:immed2}
If all $n > N$ appear as denominators between $\frac{t}{a}$ and $\frac{s}{b}$ in some $F_{m}$ with $m>\max(a,b)$, then $g(a,b) \leq N$.
\end{cor}

\begin{proof}
This follows immediately from Corollary \ref{cor:immed1} and the fact that $g(a,b)=\sup\left(\mathbb{Z} \setminus S_{\left(a,b\right)}\right)$.
\end{proof}

We note that $n$ does not appear as a denominator between $ \frac{t}{a}, \frac{s}{b}$ in any Farey sequence if and only if for every $i$ satisfying $ \frac{t}{a} < \frac{i}{n} < \frac{s}{b}$, we have $\gcd(i,n)>1$. We further note
\[\frac{t}{a} < \frac{i}{n} < \frac{s}{b} \iff \frac{t}{a}n < i < \frac{s}{b}n,\]
and it follows that the number of $i$ that must satisfy $\gcd(i,n)>1$ in order for $n$ not to appear as a denominator between $ \frac{t}{a}, \frac{s}{b}$ is $\left\lfloor \frac{sn}{b} \right\rfloor - \left\lceil \frac{tn}{a} \right\rceil+1$. We observe
\begin{align*}
    & \left\lfloor \frac{sn}{b} \right\rfloor - \left\lceil \frac{tn}{a} \right\rceil + 1
    = \left\lfloor \frac{sn}{b} \right\rfloor + \left\lfloor -\frac{tn}{a} \right\rfloor + 1
    \geq \left\lfloor \frac{sn}{b} -\frac{tn}{a} \right\rfloor
    = \left\lfloor \frac{asn-btn}{ab} \right\rfloor =\\
     & \left\lfloor \frac{n(as-bt)}{ab} \right\rfloor
    =  \left\lfloor \frac{n(1)}{ab} \right\rfloor
    =  \left\lfloor \frac{n}{ab} \right\rfloor
    \geq \frac{n}{ab}-1.
\end{align*}
This expression has the benefit of not depending on $s,t$.

It now becomes expedient to invoke the Jacobsthal function $j(n)$, which is defined for $n \geq 1$ as the smallest integer such that any sequence of $j(n)$ consecutive integers contains an element relatively prime to $n$ \cite{hajdu}. Then $j(n)-1$ is the maximum length of a string of consecutive integers \emph{not} relatively prime to $n$, and if $j(n)-1 < \frac{n}{ab}-1$ or equivalently, $j(n) < \frac{n}{ab}$, it follows by the above that $n$ is a denominator appearing between $\frac{t}{a}, \frac{s}{b}$ in Farey sequences of sufficiently high order. For 
\[ h(r):=j(p_1p_2 \cdots p_r)\]
and 
\[H(r):=\max_{\omega(n)=r} j(n),\]
Hajdu and Saradha have shown $H(r)=h(r)$ for $r \leq 23$ and $H(24) > h(24)$ \cite{hajdu}. Additionally, $h(r)$ is strictly increasing for $r \leq 63$ \cite{oeis}. We note that for all cases we have considered thus far, $N(a,b) < p_1p_2 \cdots p_{23} = 267064515689275851355624017992790$.

We combine these facts with our previous upper bound $N(a,b)$ defined in Section \ref{sec: The Existence of g(a,b)} to obtain an improved upper bound:

\begin{thm}
\label{thm:fareyub}
For $N(a,b) \leq p_1p_2\cdots p_{23} $, we have $g(a,b) \leq ab \cdot j\left(N\left(a,b\right)\right) < N(a,b)$.
\end{thm}

\begin{proof}
Recall $N(a,b)$ is defined as $p_1p_2 \cdots p_r$ for the minimum $r$ satisfying $ 2ab < \frac{p_3-1}{2} \cdot \frac{p_4-1}{2} \cdots \frac{p_r-1}{2} $, so $\omega\left(N\left(a,b\right)\right)=r$. For $m \leq N(a,b)$, we know $j(m) \leq H(r)$, and so the maximum length of any string of consecutive integers not relatively prime to $m$, that is, $j(m)-1$, is bounded above by $H(r)-1=h(r)-1=j\left(N\left(a,b\right)\right)-1$. Since 
\[ j\left(N\left(a,b\right)\right)-1 < \frac{m}{ab}-1 \iff m > ab \cdot j\left(N\left(a,b\right)\right), \]
it follows that if $m \leq N(a,b)$ and $m > ab \cdot j\left(N\left(a,b\right)\right)$, then $j(m) \leq j(N(a,b)) < \frac{m}{ab}$ and so $m$ appears as a denominator between $\frac{t}{a}, \frac{s}{b}$ in Farey sequences of sufficiently high order.

We know $ 2ab < \frac{p_3-1}{2} \cdot \frac{p_4-1}{2} \cdots \frac{p_r-1}{2} $, so 
\[ ab < \frac{1}{2} \cdot 1 \cdot \frac{p_3-1}{2} \cdot \frac{p_4-1}{2} \cdots \frac{p_r-1}{2} = \frac{p_1-1}{2} \cdot \frac{p_2-1}{2} \cdot \frac{p_3-1}{2} \cdot \frac{p_4-1}{2} \cdots \frac{p_r-1}{2}. \]
Additionally, $j\left(N\left(a,b\right)\right) \leq 2^{\omega\left(N\left(a,b\right)\right)} = 2^r$ by \cite{costello}, so
\begin{align*}
   ab \cdot j\left(N\left(a,b\right)\right) &< \frac{p_1-1}{2} \cdot \frac{p_2-1}{2} \cdot \frac{p_3-1}{2} \cdot \frac{p_4-1}{2} \cdots \frac{p_r-1}{2} \cdot 2^r\\
   &= (p_1-1)(p_2-1)\cdots (p_r-1)\\
   &< p_1p_2 \cdots p_r\\
   &= N(a,b).
\end{align*}
It follows that all $m$ satisfying $ab \cdot j\left(N\left(a,b\right)\right) < m \leq N(a,b)$ appear as denominators between $\frac{t}{a}, \frac{s}{b}$ in Farey sequences of sufficiently high order, and thus that $g(a,b) \leq ab \cdot j\left(N\left(a,b\right)\right) < N(a,b)$ by Corollary \ref{cor:immed2}.
\end{proof} 

We include Table \ref{tab:jacobs} as an illustration of the growth rate of the product of the first $r$ primes compared to the Jacobsthal function applied to the product of the first $r$ primes. We can further refine the upper bound derived using the latter as follows:

\begin{table}
    \centering
   \begin{tabular}{|c|c|c|c|c|c|c|c|c|c| }
\hline
$r$ & 1 & 2 & 3 & 4 & 5 & 6 & 7 & 8 & 9\\
\hline
$p_1p_2 \cdots p_r$ & 2 & 6 & 30 & 210 & 2310 & 30030 & 510510 & 9699690 & 223092870\\
\hline
$j\left(p_1p_2 \cdots p_r \right)$ & 2 & 4 & 6 & 10 & 14 & 22 & 26 & 34 & 40\\
\hline
\end{tabular}
    \caption{For each $r \leq 9$, the product of the first $r$ primes is given as well as the Jacobsthal function applied to these products. For $2 \leq r \leq 23$, if $p_1p_2\cdots p_r$ is an upper bound for $g(a,b)$, then $ab \cdot j\left(p_1p_2 \cdots p_r \right)$ is an improved upper bound.}
    \label{tab:jacobs}
\end{table} 
\begin{cor}
\label{cor:refinement}
For $N(a,b) \leq p_1p_2\cdots p_{23} $, we have \[g(a,b) \leq ab \cdot j\left(\frac{N\left(a,b\right)}{p_{q+1}\cdot p_{q+2}\cdots p_r}\right)\] for $q$ the maximum integer such that $p_1p_2 \cdots p_q \leq ab \cdot j\left(N\left(a,b\right)\right)$.
\end{cor}

\begin{proof}
We know by Theorem \ref{thm:fareyub} that $ ab \cdot j\left(N\left(a,b\right)\right) < N(a,b)$. Then $\omega\left(ab \cdot j\left(N\left(a,b\right)\right)\right) < r $, as $N(a,b)$ is the smallest positive integer satisfying $\omega\left(N\left(a,b\right)\right)=r$. Then for $m \leq ab \cdot j\left(N\left(a,b\right)\right)$, the maximum length of any string of consecutive integers not relatively prime to $m$ is \[ H(q)-1=h(q)-1=j\left(N\left(a,b\right)/\left(p_{q+1}\cdot p_{q+2}\cdots p_r \right)\right)-1\]
where $q$ is the maximum integer such that $p_1 \cdot p_2 \cdots p_q < ab \cdot j\left(N\left(a,b\right)\right) $. Repeating the logic of Theorem \ref{thm:fareyub}, it follows that all $m$ satisfying 
\[ab \cdot j\left(N\left(a,b\right)/\left(p_{q+1}\cdot p_{q+2}\cdots p_r \right)\right) < m < ab \cdot j\left(N\left(a,b\right)\right)\] appear as denominators between $\frac{t}{a}, \frac{s}{b}$ in Farey sequences of sufficiently high order, and thus that \[g(a,b) \leq ab \cdot j\left(N\left(a,b\right)/\left(p_{q+1}\cdot p_{q+2}\cdots p_r \right)\right) < ab \cdot j\left(N\left(a,b\right)\right)\] by Corollary \ref{cor:immed2}.
\end{proof}

We may be able to iterate Corollary \ref{cor:refinement} to further improve the upper bound, and we may be able to improve the upper bound via another method as well:

\begin{rem}
If we can show $g(a,b) \neq n$ for all $n \leq ab \cdot j\left(N\left(a,b\right)/\left(p_{q+1}\cdot p_{q+2}\cdots p_r \right)\right)$ satisfying $\omega(n)=q$ in Corollary \ref{cor:refinement}, we may replace $q$ with $q-1$ in the process above. It may be possible to iterate this step.
\label{rem: further refinement}
\end{rem}

We examine the case of $a=5, b=7$ to better illustrate the process of determining this upper bound as well as its strength. We first calculate $N(5,7)=30030$, our previous best known upper bound for $g(5,7)$. By Theorem \ref{thm:fareyub}, we now know $g(5,7) \leq 5 \cdot 7 \cdot j(30030) = 35 \cdot 22 = 770$, and by Corollary \ref{cor:refinement}, we may replace 30030 with 210 to obtain  $g(5,7) \leq 5 \cdot 7 \cdot j(210) = 350$. It takes Mathematica less than a second to determine that 210 and 330 are the only integers $n$ such that $\omega(n)=4$ and $n \leq 350$, and less than a second to verify that $5x+7y=210, 5x+7y=330$ both have solutions where $\gcd(x,y)=1$. It follows that $g(5,7) \leq 5 \cdot 7 \cdot j(30) = 5 \cdot 7 \cdot 6 =210$.  We cannot further refine the bound, as there are multiple integers $n \leq 210$ with $\omega(n)=3$ such that $5x+7y=n$ has no solutions satisfying $\gcd(x,y)=1$. By starting at $n=210$ and counting down to the first $n' \leq 210$ such that $5x+7y=n'$ has no solutions satisfying $\gcd(x,y)=1$, we obtain $g(5,7)=156$.

For the above case, we note that 156 was both the maximum integer satisfying $\omega(n)=3$ with no suitable solutions to $5x+7y=156$ and the exact value of $g(5,7)$. However, this is not always the case: the maximum integer satisfying $\omega(n)=5$ with no relatively prime integer solutions to $49x+93y=n$ is 33930, but we find $g(49,93)=40080$ and $\omega(40080)=4$. In other words, we can use this method to derive an improved upper bound, but not an exact value of $g(a,b)$. However, with an improved upper bound, we can find exact values of $g(a,b)$ more quickly through the process referenced in Section \ref{sec: Exact Values} and indicated by the code shown in Figure \ref{fig:whileloop}. Mathematica runs this code as follows: it starts with a given upper bound $N$ on $g(a,b)$. FrobeniusSolve then gives all pairs of non-negative integers $(x,y)$ that are solutions to $ax+by=N$, and GCDs computes $\gcd(x,y)$ for each pair. AnyOnes determines if $\gcd(x,y)=1$ for any pair, and the exclamation point in front means we determine if $\gcd(x,y) \neq 1$ for all pairs. When this condition is not true, that is, when there is at least one pair satisfying $\gcd(x,y)=1$, we rule out $N$ as the possible value of $g(a,b)$ and the while loop moves on to $N-1$ to try again. The first time we reach an $n$ such that $ax+by=n$ has no non-negative integer solutions satisfying $\gcd(x,y)=1$, the condition that $\gcd(x,y) \neq 1$ for all pairs is true, we know we have found $g(a,b)$, and the loop terminates. Using this code with the computing power of the University of Illinois Campus Cluster, we can derive exact values of $g(a,b)$ for all pairs $a,b \leq 200$ with $\gcd(a,b)=1$, as shown in Figure \ref{fig:fulbatchi}. 

\begin{figure}
\centering
\includegraphics[width=5.5in]{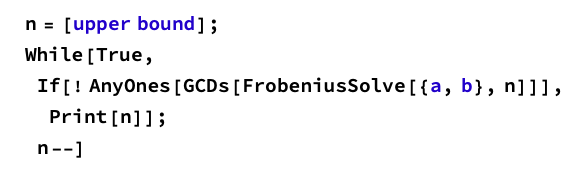}
\caption{Mathematica code that computes $g(a,b)$, given an upper bound, utilizing custom functions GCDs and AnyOnes.}
\label{fig:whileloop}
\end{figure}

\begin{figure}
    \centering
    \includegraphics[width=4.25in]{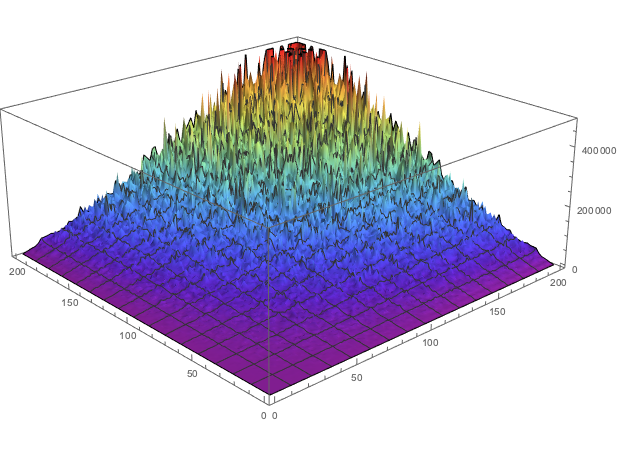}
    \caption{A plot of $g(a,b)$ for all pairs $a,b \leq 200$ with $\gcd(a,b)=1$.}
    \label{fig:fulbatchi}
\end{figure} 

\section{Future Directions}
\label{sec: future}
With the glut of new data generated in the previous section, we can search for patterns to tell us more about the behavior of $g(a,b)$. In particular, one new theorem and two new conjectures arise.

\begin{thm}
\label{multiples}
For all $\gcd(ua,b)=\gcd(a,b)=1$ and $ u > 1$, we have $g(ua,b) \geq u \cdot g(a,b)$.
\end{thm}

This is not difficult to prove by contradiction. The following conjectures hold for all data points generated thus far:

\begin{conj}
\label{conj:frob}
For all $a,b \geq 2$ relatively prime, we have $g(a,b) \geq 4 f(a,b)$.
\end{conj}

\begin{conj}
\label{conj:bstar}
For all $a,b \geq 3$ relatively prime and for $c(a,b)$ the minimum integer satisfying both $c(a,b)>b$ and $\gcd(c(a,b),a)=1$, we have $g(a,c(a,b))/g(a,b) \leq 5$.
\end{conj}

One natural future direction is to see if these conjectures can be proved or disproved. We also encounter a pattern that is somewhat too subtle to propose as a formal conjecture: define $K$ to be the set of all data points $(a,b,g(a,b))$ with $3 \leq a < b \leq 200$. Then $|K|=11933$, and by comparing 11933 divided by a prime $p$ to the number of data points for which $p$ divides $g(a,b)$, we observe that $g(a,b)$ seems to be disproportionately divisible by small primes: see Figure \ref{fig:smollprimes}. What exactly is this pattern telling us, and how might we prove what we see is true?

\begin{figure}[h]
    \centering
    \includegraphics[width=5in]{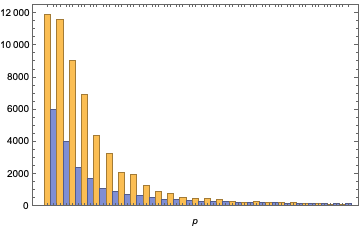}
    \caption{Along the $x$-axis, primes $p$ between 2 and 97 are plotted. In the yellow/leftmost of each pair of bars, the number of points in $K$ for which $g(a,b)$ is divisible by $p$ is plotted, and in the blue/rightmost of each pair of bars, $11933/p$, that is, the size of $K$ divided by $p$, is plotted.}
    \label{fig:smollprimes}
\end{figure}

Another avenue for future study concerns focusing on the cases where $b=3,4,6$, as there appears to be a rich but puzzling structure here. We focus on $b=3$, $a \equiv 1 \mod 3$ to illustrate our point in Figure \ref{fig:1mod3ses}. At each stage, $g(a,b)$ appears to be a roughly linear step function of $a$, with some aberrations in step size, but interestingly, we find different typical step sizes when we consider $a \leq 100$, $a \leq 1000$, $a \leq 10000$, and $a \leq 100000$ separately. What causes the step sizes to vary? Can we extrapolate what would happen if we consider $a \leq 1000000$ from the graphs given? How do we account for the atypical steps? 

\begin{figure}[h]
    \centering
    \includegraphics[width=3in]{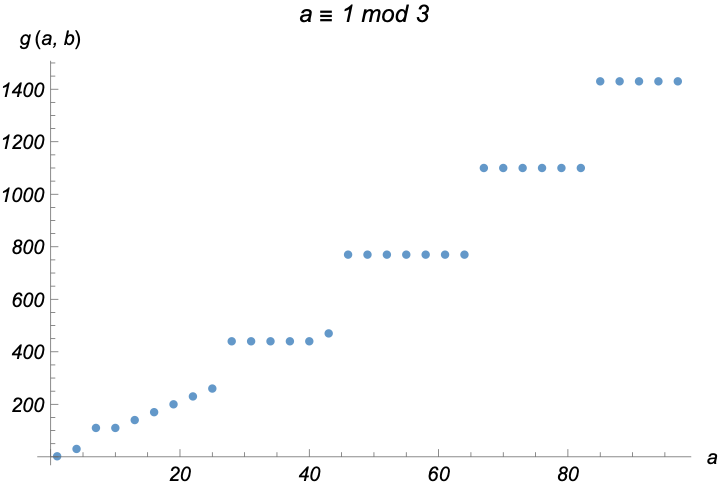}
    \includegraphics[width=3in]{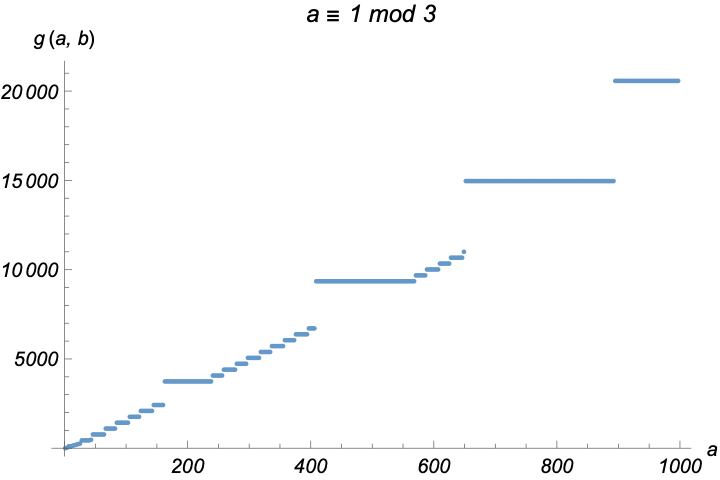}
    \includegraphics[width=3in]{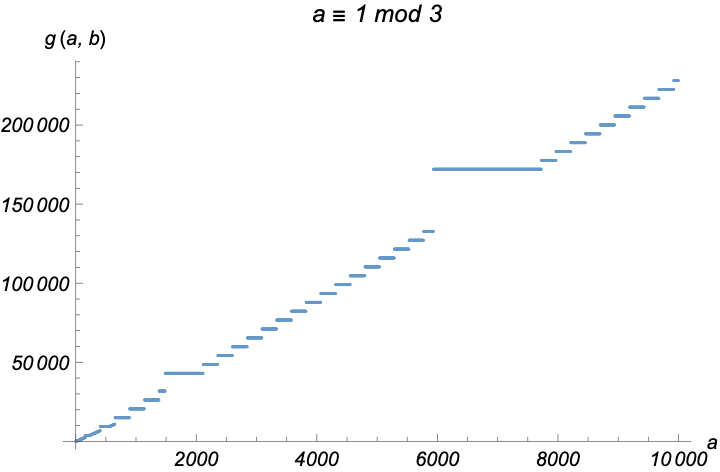}
    \includegraphics[width=3in]{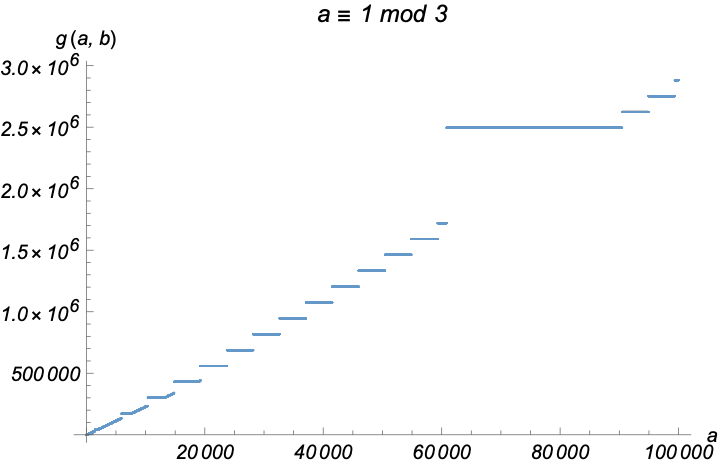}
    \caption{For $b = 3$ and $a \equiv 1 \mod 3$, $g(a,b)$ is graphed as a function of $a$ for $a \leq 100$, $a \leq 1000$, $a \leq 10000$, and $a \leq 100000$.}
    \label{fig:1mod3ses}
\end{figure}

Our focus thus far has been limited to finding $g(a,b)$, the greatest integer inexpressible as $ax+by$ under the given conditions, but what of the other inexpressible integers less than $g(a,b)$? Is there structure lurking in their distribution? This is another puzzle for which further study may prove fruitful.

\end{document}